\documentclass[titlepage,12pt]{article} 
\usepackage{xcolor}
\usepackage{amssymb,amsthm,amsmath} 
\usepackage[a4paper]{geometry}
\usepackage{datetime2}
\usepackage[utf8]{inputenc}
\usepackage{hyperref}

\date{}

\newcommand{\n}{\mathbb{N}}
\newcommand{\z}{\mathbb{Z}}

\newcommand{\re}{\mathbb{R}}

\newcommand{\ep}{\varepsilon}
\renewcommand{\qed}{{\penalty 10000\mbox{$\quad\Box$}}}

\newcommand{\meas}{\operatorname{meas}}
\newcommand{\dint}{\int\!\!\!\!\int}
\newcommand{\ud}{u_{\delta}}
\newcommand{\sn}{\mathbb{S}^{d-1}}
\newcommand{\dzp}{\delta\to 0^{+}}

\newcommand{\lz}{\Lambda_{0}}
\newcommand{\glim}{\Gamma\mbox{--}\lim}
\newcommand{\gliminf}{\Gamma\mbox{--}\liminf}

\newcommand{\pca}{\mathcal{PC\kern-1pt A}}
\newcommand{\ld}{\Lambda_{\delta}}
\newcommand{\ldh}{\widehat{\Lambda}_{\delta}}
\newcommand{\fdup}{f_{\delta}}
\newcommand{\lk}{\lambda_{k}}
\newcommand{\udh}{\widehat{u}_{\delta}}
\newcommand{\pt}{\theta}
\newcommand{\pth}{\widehat{\theta}}

\newtheorem{thm}{Theorem}[section]
\newtheorem{thmbibl}{Theorem}

\newtheorem{prop}[thm]{Proposition}
\newtheorem{defn}[thm]{Definition}
\newtheorem{cor}[thm]{Corollary}

\newtheorem{lemma}[thm]{Lemma}
\newtheorem{open}{Question}

\title{On the gap between Gamma-limit and pointwise limit for a non-local approximation of the total variation
} 

\author{Clara Antonucci\vspace{1ex}\\ 
{\normalsize Scuola Normale Superiore} \\
{\normalsize PISA (Italy)}\\
{\normalsize e-mail: \texttt{clara.antonucci@sns.it}}
\and
Massimo Gobbino\vspace{1ex}\\ 
{\normalsize Universit\`a degli Studi di Pisa} \\
{\normalsize PISA (Italy)}\\  
{\normalsize e-mail: \texttt{massimo.gobbino@unipi.it}}
\and
Nicola Picenni\vspace{1ex}\\ 
{\normalsize Scuola Normale Superiore} \\
{\normalsize PISA (Italy)}\\
{\normalsize e-mail: \texttt{nicola.picenni@sns.it}}
}

\begin{document}
\maketitle
\begin{abstract}

We consider the approximation of the total variation of a function by the family of non-local and non-convex functionals introduced by H.~Brezis and H.-M.~Nguyen in a recent paper. The approximating functionals are defined through double integrals in which every pair of points contributes according to some interaction law.

In this paper we answer two open questions concerning the dependence of the Gamma-limit on the interaction law. In the first result, we show that the Gamma-limit depends on the full shape of the interaction law, and not only on the values in a neighborhood of the origin. In the second result, we show that there do exist interaction laws for which the Gamma-limit coincides with the pointwise limit on smooth functions.

The key argument is that for some special classes of interaction laws the computation of the Gamma-limit can be reduced to studying the asymptotic behavior of suitable multi-variable minimum problems.

\vspace{6ex}

\noindent{\bf Mathematics Subject Classification 2010 (MSC2010):}
26B30, 46E35.


\vspace{6ex}


\vspace{6ex}

\end{abstract}

 
\section{Introduction}

In the recent paper~\cite{2018-AnPDE-BreNgu} (see also the note~\cite{2017-CRAS-BreNgu}, or the conference~\cite{Brezis:youtube} for a nice presentation of the topic), H.~Brezis and H.-M.~Nguyen introduced the family of non-local functionals
\begin{equation}
\ld(\varphi,u,\Omega):=\dint_{\Omega^{2}}\varphi\left(\frac{|u(y)-u(x)|}{\delta}\right)\frac{\delta}{|y-x|^{d+1}}\,dx\,dy,
\label{defn:idphi}
\end{equation}
where $d$ is a positive integer, $\Omega\subseteq\re^{d}$ is an open set, $\delta>0$ is a real parameter, $u:\Omega\to\re$ is a measurable function, and $\varphi:[0,+\infty)\to[0,+\infty)$ is a measurable function satisfying suitable properties. The function $\varphi$, whose presence is motivated by problems in image processing (see~\cite{2015-Lincei-Brezis,2018-AnPDE-BreNgu}), describes the extent to which a pair $(x,y)\in\Omega^{2}$ contributes to the double integral (\ref{defn:idphi}). For this reason, in the sequel we call $\varphi$ the ``interaction law''.

Following~\cite{2018-AnPDE-BreNgu}, we restrict ourselves to a special set of functions.

\begin{defn}[Admissible interaction laws]\label{defn:A}
\begin{em}

Let $\mathcal{A}$ denote the set of all functions $\varphi:[0,+\infty)\to[0,+\infty)$ not identically equal to zero such that
\begin{enumerate}
\renewcommand{\labelenumi}{(\roman{enumi})}

\item  $\varphi$ is nondecreasing and lower semicontinuous on $[0,+\infty)$, and actually continuous except at a finite number of points in $(0,+\infty)$,

\item  there exists a nonnegative real number $a$ such that $\varphi(t)\leq at^{2}$ for every $t\in[0,1]$,

\item  there exists a positive real number $b$ such that $\varphi(t)\leq b$ for every $t\geq 0$.

\end{enumerate}

\end{em}
\end{defn}

These conditions guarantee that when $\varphi\in\mathcal{A}$ the right-hand side of (\ref{defn:idphi}) is finite at least for every $u$ of class $C^{1}$ with compact support, and the resulting functional $\ld$ is lower semicontinuous with respect to the convergence in $L^{1}(\Omega)$.

The basic example is the case where $\varphi(x)$ coincides with
\begin{equation}
\varphi_{1}(x):=\left\{
\begin{array}{l@{\qquad}l}
 0 & \mbox{if } x\in[0,1],  \\[0.5ex]
 1 & \mbox{if } x>1.   
\end{array}
\right.
\label{defn:phi-1}
\end{equation}

In~\cite{2018-AnPDE-BreNgu} also a normalization condition is included in the definition of $\mathcal{A}$. In this paper we do not impose any normalization condition, but instead we set
\begin{equation}
N(\varphi):=\int_{0}^{+\infty}\frac{\varphi(t)}{t^{2}}\,dt
\qquad
\forall\varphi\in\mathcal{A},
\label{defn:size}
\end{equation}
and we exploit this constant as a scale factor when computing limits and Gamma-limits.
	
\paragraph{\textmd{\textit{Previous literature}}}

The asymptotic behavior of the family $\ld$ was investigated in a series of papers, starting with the model case $\varphi=\varphi_{1}$ (see~\cite{2006-CRAS-BouNgu,2006-JFA-Nguyen,2007-CRAS-Nguyen,2008-JEMS-Nguyen,2011-Duke-Nguyen}). The idea is that $\ld(\varphi,u,\Omega)$ is asymptotically proportional to the functional
\begin{equation}
\lz(u,\Omega):=\left\{
\begin{array}{l@{\qquad}l}
    \mbox{total variation of $u$ in $\Omega$}  & \mbox{if $u\in BV(\Omega)$,}  \\[1ex]
    +\infty & \mbox{if $u\in L^{1}(\Omega)\setminus BV(\Omega)$.}
\end{array}\right.
\nonumber
\end{equation}

In order to state the precise results, let $\sn:=\{\sigma\in\re^{d}:|\sigma|=1\}$ denote the unit sphere in $\re^{d}$, and let us consider the geometric constant
\begin{equation}
G_{d}:=\int_{\sn}|\langle v,\sigma\rangle|\,d\sigma,
\label{defn:gd}
\end{equation}
where $v$ is any element of $\sn$ (of course the value of $G_{d}$ does not depend on the choice of $v$), and the integration is intended with respect to the $(d-1)$-dimensional Hausdorff measure. 

The main convergence results obtained in~\cite{2018-AnPDE-BreNgu} can be summed up as follows.\begin{itemize} 

\item \emph{Pointwise convergence.} For every $\varphi\in\mathcal{A}$ it turns out that
\begin{equation}
\lim_{\dzp}\ld\left(\varphi,u,\re^d\right)=G_{d}\cdot N(\varphi)\cdot\lz\left(u,\re^{d}\right)
\qquad
\forall u\in C^{1}_{c}(\re^{d}),
\nonumber
\end{equation}
where $G_{d}$ is the geometric constant defined in (\ref{defn:gd}), and $N(\varphi)$ is the scale factor defined in (\ref{defn:size}).

On the other hand, there do exist functions $u\in W^{1,1}(\re^{d})$ for which the left-hand side is infinite (while of course the right-hand side is finite). A precise characterization of equality cases is still unknown.

\item \emph{Gamma-convergence.} For every $\varphi\in\mathcal{A}$, there exists a constant $K_{d}(\varphi)$, depending a~priori also on the space dimension, such that
\begin{equation}
\glim_{\dzp}\ld\left(\varphi,u,\re^d\right)=G_{d}\cdot N(\varphi)\cdot K_{d}(\varphi)\cdot\lz\left(u,\re^{d}\right)
\qquad
\forall u\in L^{1}(\re^{d}),
\label{thmbibl:g-conv}
\end{equation}
where the Gamma-limit is intended with respect to the usual metric of $L^{1}(\re^{d})$ (but the result would be the same with respect to the convergence in measure). 

\end{itemize}

Assuming that $K_{d}(\varphi)$ does not depend on $d$, as one reasonably expects, we could interpret (\ref{thmbibl:g-conv}) by saying that the Gamma-limit depends on the space dimension through the geometric constant $G_{d}$, on the size of $\varphi$ through the scale factor $N(\varphi)$, and on the shape of $\varphi$ through the ``shape factor'' $K_{d}(\varphi)$. 

The behavior under rescaling clarifies the different nature of the scale and the shape factors. If we replace $\varphi(t)$ by $\widehat{\varphi}(t):=\alpha\varphi(\beta t)$ for some positive constants $\alpha$ and $\beta$, a change of variables shows that $N(\widehat{\varphi})=\alpha\beta N(\varphi)$, and the same scaling affects the left-hand side of (\ref{thmbibl:g-conv}). It follows that $K_{d}(\varphi)=K_{d}(\widehat{\varphi})$, namely the shape factor is invariant by both horizontal and vertical rescaling.

Very little was known about $K_{d}(\varphi)$. In~\cite{2007-CRAS-Nguyen,2011-Duke-Nguyen} it was proved that $K_{d}(\varphi_{1})\leq\log 2$, where $\varphi_{1}$ is the model interaction law defined in (\ref{defn:phi-1}). In~\cite{2018-AnPDE-BreNgu} it was proved that
\begin{equation}
K_{d}(\varphi_{1})\leq K_{d}(\varphi)\leq 1
\qquad
\forall\varphi\in\mathcal{A}.
\nonumber
\end{equation}

In order to shed some light on $K_{d}(\varphi)$, whose appearance in the Gamma-limit was defined in~\cite{2018-AnPDE-BreNgu} ``mysterious and somewhat counterintuitive'', some open questions were explicitly stated. The first one addresses the dependence of $K_{d}(\varphi)$ on the full shape of~$\varphi$.

\begin{open}[{see~\cite[Open Problem~4]{2018-AnPDE-BreNgu}}]\label{Q4}

Assume that two functions $\varphi$ and $\psi$ in $\mathcal{A}$ satisfy $N(\varphi)=N(\psi)$, and
$$\varphi\geq\psi \mbox{ near 0}\quad (\mbox{resp. }\varphi=\psi \mbox{ near 0}\,).$$ 

Is it true that
$$K_{d}(\varphi)\geq K_{d}(\psi)\quad \left(\mbox{resp. }K_{d}(\varphi)= K_{d}(\psi)\strut\right)?$$ 

\end{open}

A positive answer to this question would imply that, once that the scale factor has been fixed, only the behavior of the interaction law in a neighborhood of the origin is relevant to the Gamma-limit. This intuition is supported by the observation that the kernel in (\ref{defn:idphi}) is divergent on the diagonal $y=x$, and therefore short-range interactions should be more relevant in the computation of $K_{d}(\varphi)$.

The second question addresses the necessity and the width of the gap between the pointwise limit and the Gamma-limit.

\begin{open}[{see~\cite[Open Problem~2]{2018-AnPDE-BreNgu}}]\label{Q2}

Is it true that $K_{d}(\varphi)<1$ for every $\varphi\in\mathcal{A}$ (and every space dimension $d$)? Or even better: Is it true that 
\begin{equation}
\sup\{K_{d}(\varphi):\varphi\in\mathcal{A}\}<1\mbox{\,?}
\label{th:Q2-strong}
\end{equation}

\end{open}

A positive answer to this question, especially in the stronger form (\ref{th:Q2-strong}), would imply that the counterintuitive gap is structural.

\paragraph{\textmd{\textit{Our results}}}

In this paper we show that the answer to both questions is negative. To this end, we introduce the following special classes of interaction laws.

\begin{defn}[Special interaction laws]
\begin{em}

For every positive integer $k$, we consider the interaction law $\varphi_{k}:[0,+\infty)\to[0,+\infty)$ defined as
\begin{equation}
\varphi_{k}(t):=\varphi_{1}\left(\frac{t}{k}\right)=\left\{
\begin{array}{l@{\qquad}l}
 0 & \mbox{if } t\in[0,k],  \\[0.5ex]
 1 & \mbox{if } t>k.
\end{array}
\right.
\label{defn:phi-k}
\end{equation}

\begin{itemize}

\item  Let $\mathcal{A}_{0}$ denote the set of interaction laws that vanish in $[0,1]$, namely
\begin{equation}
\mathcal{A}_{0}:=\{\varphi\in\mathcal{A}:\varphi(t)=0\quad \forall t\in[0,1]\}.
\end{equation}

\item  Let $\pca$ denote the set of interaction laws that can be written in the form
\begin{equation}
\varphi(t)=\sum_{k=1}^{m}\lk\varphi_{k}(t)
\qquad
\forall t\geq 0
\label{defn:phi-pca}
\end{equation}
for some positive integer $m$, and some nonnegative real numbers $\lambda_{1}$, \ldots, $\lambda_{m}$ (not all equal to zero).

\item  Let $\pca_{2}$ denote the set of interaction laws of the form (\ref{defn:phi-pca}) whose coefficients  are equal in packages of powers of two, namely 
\begin{equation}
\lambda_{2}=\lambda_{3},
\qquad\quad
\lambda_{4}=\ldots=\lambda_{7},
\qquad\quad
\lambda_{8}=\ldots=\lambda_{15},
\nonumber
\end{equation}
and so on. More precisely, every $\varphi\in\pca_{2}$ can be written in the form
\begin{equation}
\varphi(t):=\sum_{j=1}^{m}\left(a_{j}\sum_{k=2^{j-1}}^{2^{j}-1}\varphi_{k}(t)\right)
\qquad
\forall t\geq 0
\label{defn:pca2}
\end{equation}
for some positive integer $m$, and some nonnegative real numbers $a_{1}$, \ldots, $a_{m}$ (not all equal to zero).

\end{itemize}

\end{em}
\end{defn}

We observe that
$\mathcal{A}\supseteq\mathcal{A}_{0}\supseteq\pca\supseteq\pca_{2}$,
and all inclusions are strict.

Our first result is the following.

\begin{thm}[Piecewise constant interaction laws]\label{thm:Q4}

Let $K_{d}(\varphi)$ be the shape factor of an interaction law as defined by (\ref{thmbibl:g-conv}). 

Then in every space dimension $d$ it turns out that
\begin{equation}
\sup\{K_{d}(\varphi):\varphi\in\pca_{2}\}=1.
\nonumber
\end{equation}

\end{thm}

A closer look at the proof reveals that the supremum is realized for example by the interaction laws of the form (\ref{defn:pca2}) with $a_{1}=\ldots=a_{m}=1$, in the limit as $m\to +\infty$. 

Theorem~\ref{thm:Q4} above provides a negative answer to Question~\ref{Q4}, as well as a negative answer to Question~\ref{Q2} in the stronger form (\ref{th:Q2-strong}). In particular, this means that the full shape of the interaction law comes into play in the computation of the Gamma-limit, which therefore keeps into account both short-range and long-range interactions. At the beginning of Section~\ref{sec:speculation} we present also an example with strict inequalities, namely with $\varphi>\psi$ near the origin, but $K_{d}(\varphi)<K_{d}(\psi)$.

Our second result gives a stronger negative answer to Question~\ref{Q2}, even when restricted to the smaller class $\mathcal{A}_{0}$.

\begin{thm}[Piecewise affine dyadic interaction laws]\label{thm:Q2}

Let $K_{d}(\varphi)$ be the shape factor of an interaction law as defined by (\ref{thmbibl:g-conv}). 

Then the following statements hold true.

\begin{enumerate}
\renewcommand{\labelenumi}{(\arabic{enumi})}

\item In every space dimension $d$ it turns out that
\begin{equation}
\max\{K_{d}(\varphi):\varphi\in\mathcal{A}_{0}\}=1.
\nonumber
\end{equation}

\item More precisely, let $f:\z\to[0,+\infty)$ be any nondecreasing and bounded function (not identically equal to zero) such that
\begin{equation}
\limsup_{n\to+\infty}f(-n)\cdot 4^{n}<+\infty.
\label{hp:f-quadratic}
\end{equation}

Let us consider the function $\zeta:[0,+\infty)\to[0,+\infty)$ such that
\begin{itemize}
\item $\zeta(0)=0$,
\item $\zeta(2^{z})=f(z)$ for every $z\in\z$,
\item $\zeta$ is affine in the interval $[2^{z},2^{z+1}]$ for every $z\in\z$.
\end{itemize}

Then it turns out that $\zeta\in\mathcal{A}$, and $K_{d}(\zeta)=1$ in every space dimension $d$.

\end{enumerate}

\end{thm}

The second statement of Theorem~\ref{thm:Q2} above shows in particular that there are large classes of interaction laws for which the Gamma-limit of (\ref{defn:idphi}) coincides with the pointwise limit for smooth functions.

\paragraph{\textmd{\textit{Overview of the technique}}}

The proof of these results follows the same strategy that in~\cite{AGMP:log-2} led us to show that actually $K_{d}(\varphi_{1})=\log 2$ (see also the note~\cite{AGMP:CRAS} for an informal summary of our approach).  Since we only need estimates from below for the shape factor, we can limit ourselves to estimating from below the Gamma-liminf. The main steps are the following.
\begin{itemize}

\item \emph{From local to global bounds}. In Theorem~\ref{prop:loc2glob} we reduce the problem in any dimension to intervals of the real line, namely to showing that for $\delta$ small enough we can estimate from below $\ld(\varphi,u,(a,b))$ in terms of the oscillation of $u$ in $(a,b)$.

\item \emph{Reduction to multi-variable minimum problems}. In Proposition~\ref{prop:main} we show that, when $\varphi\in\pca$, we can assume that $u$ is a nondecreasing step function with finite image contained in $\delta\z$. Up to vertical translations, any such function depends only on the lengths $\ell_{1},\ldots,\ell_{n}$ of the steps, where $n\sim\delta^{-1}$. In this way we reduce ourselves to studying the minimum of a $\delta$-dependent multi-variable function.

\item \emph{Telescopic effect}. The minimum problems found in Proposition~\ref{prop:main} can be very complicated. Nevertheless, in the special case where $\varphi\in\pca_{2}$, a telescopic effect implies that the configurations with $\ell_{1}=\ldots=\ell_{n}$ are asymptotically optimal (see Proposition~\ref{prop:package}). At this point, we have explicit estimates from below for the Gamma-liminf, which in turn yield explicit estimates from below for shape factors, thus leading to the proof of Theorem~\ref{thm:Q4}.

\item \emph{Approximation from below of special interaction laws}. The more general interaction laws of Theorem~\ref{thm:Q2} can be approximated from below by sequences of interaction laws that, up to horizontal rescaling, are in $\pca_{2}$. Again this provides estimates from below for the Gamma-liminf, and hence for shape factors, that lead to the proof of Theorem~\ref{thm:Q2}.

\end{itemize}
	
\paragraph{\textmd{\textit{Structure of the paper}}}

This paper is organized as follows. In Section~\ref{sec:previous} we recall the technical results from~\cite{AGMP:log-2} that are needed in the sequel. In Section~\ref{sec:min-pbm} we investigate the asymptotic behavior of suitable sequences of minimum problems for functions of a finite number of variables. In Section~\ref{sec:cost} we prove Proposition~\ref{prop:main}, which establishes the connection between the Gamma-limit of (\ref{defn:idphi}) and the minimum problems of the previous section. In Section~\ref{sec:proofs} we prove Theorem~\ref{thm:Q4} and Theorem~\ref{thm:Q2}. In Section~\ref{sec:speculation} we speculate about some possible generalizations and future directions.


\setcounter{equation}{0}
\section{Preliminary results}\label{sec:previous}

In this section we collect, for the convenience of the reader, the results from~\cite{AGMP:log-2} that are crucial to this paper. To begin with, we recall the definitions of truncation, vertical $\delta$-segmentation and nondecreasing rearrangement.

\begin{defn}[Truncation]\label{defn:truncation}
\begin{em}

Let $\mathbb{X}$ be any set, let $w:\mathbb{X}\to\re$ be any function, and let $A<B$ be two real numbers. The truncation of $w$ between $A$ and $B$ is the function $T_{A,B}w:\mathbb{X}\to\re$ defined by
$$T_{A,B}w(x):=\left\{
\begin{array}{l@{\qquad}l}
   A   & \mbox{if }w(x)<A,   \\[0.5ex]
   w(x)   & \mbox{if }A\leq w(x)\leq B,   \\[0.5ex]
   B   & \mbox{if }w(x)>B.  
\end{array}
\right.$$

\end{em}
\end{defn}

\begin{defn}[Vertical $\delta$-segmentation]\label{defn:delta-segm}
\begin{em}

Let $\mathbb{X}$ be any set, let $w:\mathbb{X}\to\re$ be any function, and let $\delta$ be a positive real number. The vertical $\delta$-segmentation of $w$ is the function $S_{\delta}w:\mathbb{X}\to\re$ defined by
\begin{equation}
S_{\delta}w(x):=\delta\left\lfloor\frac{w(x)}{\delta}\right\rfloor
\qquad
\forall x\in\mathbb{X}.
\nonumber
\end{equation}

The function $S_{\delta}w$ takes its values in $\delta\z$, and it is uniquely characterized by the fact that $S_{\delta}w(x)=k\delta$ for some $k\in\z$ if and only if $k\delta\leq w(x)<(k+1)\delta$.

\end{em}
\end{defn}

\begin{defn}[Nondecreasing rearrangement]\label{defn:rearrangement}
\begin{em}

Let $(a,b)\subseteq\re$ be an interval, and let $w:(a,b)\to\re$ be a function whose image is a finite set. The nondecreasing rearrangement of $w$ is the function $Mw:(a,b)\to\re$ defined by
\begin{equation}
Mw(x):=\min\left\{y\in\re:\meas\{z\in(a,b):w(z)\leq y\}\geq x-a\strut\right\}
\qquad
\forall x\in\re.
\label{defn:Mu}
\end{equation}

As expected, the function $Mw$ is nondecreasing, and satisfies
$$\meas\{x\in(a,b):Mw(x)=y\}=\meas\{x\in(a,b):w(x)=y\}
\qquad
\forall y\in\re.$$

\end{em}
\end{defn}


\subsection{Semi-discrete aggregation problem}

Let us recall the main result of~\cite[Section~2]{AGMP:log-2}, where a class of discrete and semi-discrete aggregation problems was considered. Following the terminology introduced therein, the basic ingredients are an interval $(a,b)\subseteq\re$, an integer number $k\geq 1$, and  a nonincreasing function $c:(0,b-a)\to\re$, possibly unbounded in a neighborhood of the origin, called hostility function. A semi-discrete arrangement is any measurable function $u:(a,b)\to\z$ with finite image. 

For any such function $u$, we define the total $k$-hostility as
\begin{equation}
\mathcal{F}_{k}(c,u)=\dint_{(a,b)^{2}}\varphi_{k}(|u(y)-u(x)|)\cdot c(|y-x|)\,dx\,dy,
\label{dino:idphi}
\end{equation}
where $\varphi_{k}(t)$ is the interaction law defined by (\ref{defn:phi-k}).

The key result proved in~\cite[Theorem~2.4]{AGMP:log-2} (and equivalent to some rearrangement inequalities found independently in~\cite{1973-JCT-Taylor,1974-AnIF-GarRod} in a different context) is that the nondecreasing rearrangement does not increase the total hostility. 

\begin{thmbibl}[Total hostility minimization]\label{thm:dino-sd}

Let $(a,b)\subseteq\re$ be an interval, let $k\geq 1$ be an integer, and let $c:(0,b-a)\to\re$ be a nonincreasing function. Let $u:(a,b)\to\z$ be a measurable function with finite image, let $Mu:(a,b)\to\z$ be its nondecreasing rearrangement defined by (\ref{defn:Mu}), and let $\mathcal{F}_{k}(c,u)$ be the functional defined by (\ref{dino:idphi}).

Then it turns out that
\begin{equation}
\mathcal{F}_{k}(c,u)\geq\mathcal{F}_{k}(c,Mu).
\nonumber
\end{equation}
\end{thmbibl}


\subsection{Localization technique}

One of the main points in~\cite{AGMP:log-2} was obtaining a localized version of the Gamma-liminf inequality, namely an asymptotic estimate from below for $\ld(\varphi,\ud,(a,b))$ in terms of the oscillation of $\ud$ in $(a,b)$. After such an estimate has been established, a quite classical path (see for example~\cite{ms,tesi-mora}), independent of the presence of the interaction law $\varphi$, leads to an estimate from below for the Gamma-liminf of $\ld$ in any space dimension, and hence to an estimate from below for the shape factor of $\varphi$. The precise statement is the following.

\begin{thmbibl}[From local to global bounds]\label{prop:loc2glob}

Let $\varphi\in\mathcal{A}$ be an interaction law, with scale factor $N(\varphi)$ defined by (\ref{defn:size}), and shape factor $K_{d}(\varphi)$ defined by (\ref{thmbibl:g-conv}). 

Let us assume that there exists a constant $K_{0}$ such that, for every interval $(a,b)\subseteq\re$, and every family $\{\ud\}_{\delta>0}\subseteq L^{1}((a,b))$, it happens that
\begin{equation}
\liminf_{\dzp}\ld(\varphi,\ud,(a,b))\geq K_{0}\cdot\liminf_{\dzp}\operatorname{osc}(\ud,(a,b)),
\label{hp:loc2glob}
\end{equation}
where $\operatorname{osc}(\ud,(a,b))$ denotes the essential oscillation of $\ud$ in $(a,b)$.

Then for every positive integer $d$ it turns out that
\begin{equation}
\gliminf_{\dzp}\ld\left(\varphi,u,\re^{d}\right)\geq G_{d}\cdot\frac{K_{0}}{2}\cdot\lz\left(u,\re^{d}\right)
\qquad
\forall u\in L^{1}(\re^{d}).
\label{th:loc2glob}
\end{equation}

\end{thmbibl}

The proof of Theorem~\ref{prop:loc2glob} has two distinct steps. 

The first one, for which we refer to~\cite[Section~3.2]{AGMP:log-2}, exploits a piecewise affine approximation in order to deduce (\ref{th:loc2glob}) in dimension one from the local estimate (\ref{hp:loc2glob}). We remark that $G_{1}=2$, and therefore in dimension one the geometric constant cancels the denominator, so that (\ref{hp:loc2glob}) is exactly the localized version of (\ref{th:loc2glob}).

The second step, for which we refer to~\cite[Section~4]{AGMP:log-2}, relies on an integral-geometric representation of both the total variation and the double integral (\ref{defn:idphi}). This representation leads from (\ref{th:loc2glob}) in dimension one to the analogous inequality in any space dimension. 


\setcounter{equation}{0}
\section{A family of multi-variable minimum problems}\label{sec:min-pbm}

Let $n$ be a positive integer, and let $(\ell_{1},\ldots,\ell_{n})$ be an $n$-tuple of nonnegative real numbers. For every positive integer $k\leq n$ we consider all possible sums of $k$ consecutive terms
\begin{equation}
S_{i,k}(\ell_{1},\ldots,\ell_{n}):=\sum_{h=0}^{k-1}\ell_{i+h}
\qquad
\forall i\in\{1,\ldots,n-k+1\},
\label{defn:Sik}
\end{equation}
and we define the set
\begin{equation}
D_{n,k}:=\left\{(\ell_{1},\ldots,\ell_{n})\in[0,+\infty)^{n}:S_{i,k}(\ell_{1},\ldots,\ell_{n})>0\quad\forall i\in\{1,\ldots,n-k+1\}\strut\right\}
\end{equation}
of all $n$-tuples of nonnegative real numbers without $k$ consecutive components equal to zero. When $n\geq k+1$, for every $(\ell_{1},\ldots,\ell_{n})\in D_{n,k}$ we can set
\begin{equation}
L_{k}(\ell_{1},\ldots,\ell_{n}):=\sum_{i=1}^{n-k}\log\frac{[S_{i,k+1}(\ell_{1},\ldots,\ell_{n})]^{2}}{S_{i,k}(\ell_{1},\ldots,\ell_{n})\cdot S_{i+1,k}(\ell_{1},\ldots,\ell_{n})}.
\label{defn:Lk}
\end{equation}

Given any interaction law $\varphi\in\pca$, we call $\mu(\varphi)$ the smallest integer $k$ such that $\lambda_{k}\neq 0$ in the representation (\ref{defn:phi-pca}), and for every $n\geq m+1$ we consider the homogeneous function $P_{n,\varphi}:D_{n,\mu(\varphi)}\to\re$ defined by
\begin{equation}
P_{n,\varphi}(\ell_{1},\ldots,\ell_{n}):=\sum_{k=\mu(\varphi)}^{m}\lk L_{k}(\ell_{1},\ldots,\ell_{n})
\qquad
\forall (\ell_{1},\ldots,\ell_{n})\in D_{n,\mu(\varphi)},
\label{defn:pnp}
\end{equation}
and its infimum
\begin{eqnarray}
I_{n}(\varphi) & := & \inf\left\{P_{n,\varphi}(\ell_{1},\ldots,\ell_{n}):(\ell_{1},\ldots,\ell_{n})\in D_{n,\mu(\varphi)}\strut\right\}
\nonumber \\[0.5ex]
 & = & \inf\left\{P_{n,\varphi}(\ell_{1},\ldots,\ell_{n}):(\ell_{1},\ldots,\ell_{n})\in(0,+\infty)^{n}\strut\right\}.
\label{defn:Inp}
\end{eqnarray}

Let us consider for example the interaction law $\varphi_{1}$ defined in (\ref{defn:phi-1}). In this case $\mu(\varphi_{1})=1$, the set $D_{n,1}$ is just $(0,+\infty)^{n}$, and the function in (\ref{defn:pnp}) has the form
\begin{eqnarray}
P_{n,\varphi_{1}}(\ell_{1},\ldots,\ell_{n}) & = & L_{1}(\ell_{1},\ldots,\ell_{n}) 
\nonumber  \\[1ex]
& = & \log\frac{(\ell_{1}+\ell_{2})^{2}}{\ell_{1}\ell_{2}}+\log\frac{(\ell_{2}+\ell_{3})^{2}}{\ell_{2}\ell_{3}}+\ldots+\log\frac{(\ell_{n-1}+\ell_{n})^{2}}{\ell_{n-1}\ell_{n}}.
\nonumber
\end{eqnarray}

All the fractions inside the logarithms are greater than or equal to 4, and hence $I_{n}(\varphi_{1})=(n-1)\log 4$, with the minimum realized when all the variables are equal. This computation was the final step in the proof of the crucial estimate in~\cite{AGMP:log-2}. In Section~\ref{sec:cost} of the present paper we show that the asymptotic behavior of $I_{n}(\varphi)$ plays a fundamental role in estimating from below the shape factor of any interaction law $\varphi\in\pca$. 

Unfortunately, things are not that simple for larger values of $k$. For example, when the interaction law is $\varphi_{3}$, the function $L_{3}(\ell_{1},\ldots,\ell_{n})$ is given by
\begin{equation}
\log\frac{(\ell_{1}+\ell_{2}+\ell_{3}+\ell_{4})^{2}}{(\ell_{1}+\ell_{2}+\ell_{3})(\ell_{2}+\ell_{3}+\ell_{4})}+\log\frac{(\ell_{2}+\ell_{3}+\ell_{4}+\ell_{5})^{2}}{(\ell_{2}+\ell_{3}+\ell_{4})(\ell_{3}+\ell_{4}+\ell_{5})}+\ldots,
\nonumber
\end{equation}
and now the minimum is not realized when all the variables are equal (for example the periodic pattern $1,0,0,1,0,0,\ldots$ is better). 

Of course any interaction law of the form $\varphi_{k}$ can be dealt with as a rescaling of $\varphi_{1}$, but nevertheless the appearance of different patterns in the minimization process seems to suggest that things get worse and worse when we take linear combinations of the form (\ref{defn:pnp}). Fortunately this is not always the case. Indeed, when we expand linear combinations of this form, the numerators of the terms of $L_{k}$ can partially cancel with the denominators of the terms of $L_{k+1}$, leading to the following result.

\begin{lemma}[Telescopic effect]\label{lemma:telescopic}

Let $a$, $b$, and $n$ be positive integers such that $a\leq b\leq n-1$. Let $S_{i,k}$ and $L_{k}$ be the functions of $n$ variables defined in (\ref{defn:Sik}) and (\ref{defn:Lk}).

Then for every $(\ell_{1},\ldots,\ell_{n})\in D_{n,a}$ it turns out that
\begin{equation}
\sum_{j=a}^{b}L_{j}(\ell_{1},\ldots,\ell_{n})\geq\sum_{i=1}^{n-b}\log\frac{[S_{i,b+1}(\ell_{1},\ldots,\ell_{n})]^{2}}{S_{i,a}(\ell_{1},\ldots,\ell_{n})\cdot S_{i+(b-a)+1,a}(\ell_{1},\ldots,\ell_{n})}.
\label{th:telescopic}
\end{equation}

\end{lemma}

\begin{proof}

To begin with, we observe that (\ref{th:telescopic}) is an equality when $b=a$. Therefore, in the sequel we assume that $b\geq a+1$. For the sake of shortness, throughout this proof we omit the explicit dependence on the variables $\ell_{1}$, \ldots, $\ell_{n}$. The left-hand side of (\ref{th:telescopic}) can be written in the form
\begin{equation}
\sum_{j=a}^{b}L_{j}=2\Sigma_{1}-\Sigma_{2}-\Sigma_{3},
\label{sigma-1-2-3}
\end{equation}
where
\begin{equation}
\Sigma_{1}:=\sum_{j=a}^{b}\sum_{i=1}^{n-j}\log S_{i,j+1},
\qquad
\Sigma_{2}:=\sum_{j=a}^{b}\sum_{i=1}^{n-j}\log S_{i,j},
\qquad
\Sigma_{3}:=\sum_{j=a}^{b}\sum_{i=1}^{n-j}\log S_{i+1,j}.
\nonumber
\end{equation}

With some algebra (shift of indices and separation of the terms corresponding to the first or last value of some index) we can rewrite the three sums as follows
\begin{equation}
\Sigma_{1}=\sum_{i=1}^{n-b}\log S_{i,b+1}+\sum_{j=a+1}^{b}\log S_{1,j}+\sum_{j=a+1}^{b}\log S_{n-j+1,j}+\sum_{j=a+1}^{b}\sum_{i=2}^{n-j}\log S_{i,j},
\nonumber
\end{equation}
\begin{equation}
\Sigma_{2}=\sum_{i=1}^{n-a}\log S_{i,a}+\sum_{j=a+1}^{b}\log S_{1,j}+\sum_{j=a+1}^{b}\sum_{i=2}^{n-j}\log S_{i,j},
\nonumber
\end{equation}
\begin{equation}
\Sigma_{3}=\sum_{i=2}^{n-a+1}\log S_{i,a}+\sum_{j=a+1}^{b}\log S_{n-j+1,j}+\sum_{j=a+1}^{b}\sum_{i=2}^{n-j}\log S_{i,j}.
\nonumber
\end{equation}

When we plug these three equalities into (\ref{sigma-1-2-3}), all double sums cancel, and also the second sums in $\Sigma_{2}$ and $\Sigma_{3}$ cancel with a part of the second and third term in $\Sigma_{1}$. We end up with
\begin{eqnarray}
2\Sigma_{1}-\Sigma_{2}-\Sigma_{3} & = & 2\sum_{i=1}^{n-b}\log S_{i,b+1}+\sum_{j=a+1}^{b}\log S_{1,j}+\sum_{j=a+1}^{b}\log S_{n-j+1,j}  
\nonumber  \\[1ex]
& & -\sum_{i=1}^{n-a}\log S_{i,a}-\sum_{i=2}^{n-a+1}\log S_{i,a}.
\label{sigma-bis}
\end{eqnarray}

Let us reorganize these terms. In the second and third sum we change the indices, and we rewrite them as
\begin{equation}
\sum_{j=a+1}^{b}\log S_{1,j}=\sum_{k=1}^{b-a}\log S_{1,k+a},
\nonumber
\end{equation}
\begin{equation}
\sum_{j=a+1}^{b}\log S_{n-j+1,j}=\sum_{k=n-b+1}^{n-a}\log S_{k,n+1-k}.
\nonumber
\end{equation}

In the fourth sum we split the terms as follows
\begin{equation}
\sum_{i=1}^{n-a}\log S_{i,a}=\sum_{i=1}^{n-b}\log S_{i,a}+\sum_{k=n-b+1}^{n-a}\log S_{k,a}.
\nonumber
\end{equation}

In the fifth sum we split the terms, and then we shift one index in order to rewrite the sum as
\begin{eqnarray}
\sum_{i=2}^{n-a+1}\log S_{i,a} & = & \sum_{i=2}^{b-a+1}\log S_{i,a}+\sum_{i=b-a+2}^{n-a+1}\log S_{i,a} 
\nonumber  \\[1ex]
 & = & \sum_{k=1}^{b-a}\log S_{k+1,a}+\sum_{i=1}^{n-b}\log S_{i+(b-a)+1,a}.
\nonumber
\end{eqnarray}

Plugging all these equalities into (\ref{sigma-bis}) we find that
\begin{eqnarray}
2\Sigma_{1}-\Sigma_{2}-\Sigma_{3} & = & \sum_{i=1}^{n-b}\log\frac{[S_{i,b+1}]^{2}}{S_{i,a}\cdot S_{i+(b-a)+1,a}} 
\nonumber  \\[1ex]
 & & +\sum_{k=1}^{b-a}\log\frac{S_{1,k+a}}{S_{k+1,a}}+\sum_{k=n-b+1}^{n-a}\log\frac{S_{k,n+1-k}}{S_{k,a}}.
\nonumber
\end{eqnarray}

In the sums of the last line, all terms are nonnegative because in all the fractions the numerators are greater than or equal to the corresponding denominators. Recalling (\ref{sigma-1-2-3}), it follows that
\begin{equation}
\sum_{j=a}^{b}L_{j}=2\Sigma_{1}-\Sigma_{2}-\Sigma_{3}\geq\sum_{i=1}^{n-b}\log\frac{[S_{i,b+1}]^{2}}{S_{i,a}\cdot S_{i+(b-a)+1,a}},
\nonumber
\end{equation}
which completes the proof of (\ref{th:telescopic}).
\end{proof}


\begin{cor}\label{cor:2}

Let us consider the situation described in Lemma~\ref{lemma:telescopic} in the special case where $a=2^{m-1}$ and $b=2^{m}-1$ for some positive integer $m$.

Then for every $n\geq 2^{m}$ it turns out that
\begin{equation}
\sum_{k=2^{m-1}}^{2^{m}-1}L_{k}(\ell_{1},\ldots,\ell_{n})\geq (n-2^{m}+1)\cdot 2\log 2
\qquad
\forall(\ell_{1},\ldots,\ell_{n})\in D_{n,a}.
\label{th:cor-2}
\end{equation}

\end{cor}

\begin{proof}

In this special case it turns out that
\begin{equation}
S_{i,b+1}=S_{i,a}+S_{i+(b-a)+1,a}
\qquad
\forall i\leq n-2^{m}+1.
\nonumber
\end{equation}

Therefore, from the inequality between arithmetic mean and geometric mean, we deduce that all the fractions in the right-hand side of (\ref{th:telescopic}) are greater than or equal to 4, and this is enough to establish (\ref{th:cor-2}).
\end{proof}


From Corollary~\ref{cor:2} we deduce a lower bound for the asymptotic behavior of $I_{n}(\varphi)$ for interaction laws $\varphi\in\pca_{2}$.

\begin{prop}[Interaction laws with package structure]\label{prop:package}

Let $m$ be a positive integer, let $a_{1}$, \ldots, $a_{m}$ be nonnegative real numbers (not all equal to 0), and let $\varphi\in\pca_{2}$ be defined as in (\ref{defn:pca2}).
 
For every integer $n\geq 2^{m}$, let $P_{n,\varphi}$ be the homogeneous function defined by (\ref{defn:pnp}), and let $I_{n}(\varphi)$ be its infimum as in (\ref{defn:Inp}).

Then it turns out that
\begin{equation}
\liminf_{n\to +\infty}\frac{I_{n}(\varphi)}{n}\geq 2\log 2\cdot\sum_{j=1}^{m}a_{j}.
\label{th:lim-in}
\end{equation}
\end{prop}

\begin{proof}

Let $m_{0}(\varphi)$ denote the smallest integer $k$ such that $a_{k}>0$. From Corollary~\ref{cor:2} we know that
\begin{equation}
\sum_{k=2^{j-1}}^{2^{j}-1}L_{k}(\ell_{1},\ldots,\ell_{n})\geq(n-2^{j}+1)\cdot 2\log 2
\nonumber
\end{equation}
for every $j\in\{m_{0}(\varphi),\ldots,m\}$, and therefore
\begin{equation}
P_{n,\varphi}(\ell_{1},\ldots,\ell_{n})\geq\sum_{j=m_{0}(\varphi)}^{m}a_{j}\cdot(n-2^{j}+1)\cdot 2\log 2\geq(n-2^{m})\cdot 2\log 2\cdot\sum_{j=1}^{m}a_{j}
\nonumber
\end{equation}
for every admissible choice of $\ell_{1}$, \ldots, $\ell_{n}$.  

Dividing by $n$, and letting $n\to+\infty$, we obtain (\ref{th:lim-in}).
\end{proof}


\setcounter{equation}{0}
\section{Asymptotic cost of oscillations}\label{sec:cost}

In this section we clarify the connection between the Gamma-limit of the family (\ref{defn:idphi}) and the multi-variable functions of Section~\ref{sec:min-pbm}. In analogy with~\cite{AGMP:log-2}, the question we address is the following. Let us assume that a function $\ud(x)$ oscillates between two values $A$ and $B$ in some interval $(a,b)$. Does this provide an estimate from below for $\ld(\varphi,\ud,(a,b))$, at least  when $\delta$ is small enough? A quantitative answer is provided by Proposition~\ref{prop:main} and Corollary~\ref{cor:osc}, and this answer is connected  to the Gamma-limit of the family (\ref{defn:idphi}) by Theorem~\ref{prop:loc2glob}, as we clarify in Proposition~\ref{prop:g-liminf-pca}. 

To begin with, we show that three simplifying operations can be performed on $\ud$ without neither changing its oscillation between $A$ and $B$, nor increasing its energy.


\begin{lemma}[Truncation, segmentation, rearrangement]\label{lemma:MST}

Let $a<b$ and $A<B$ be real numbers, let $u:(a,b)\to\re$ be a measurable function, and let $\varphi\in\pca$.

Then for every $\delta>0$ it turns out that
\begin{equation}
\ld(\varphi,u,(a,b))\geq\ld(\varphi,MS_{\delta}T_{A,B}u,(a,b)),
\label{th:MST}
\end{equation}
where $T_{A,B}$, $S_{\delta}$, and $M$ are the operators of truncation, vertical $\delta$-segmentation, and nondecreasing rearrangement defined at the beginning of Section~\ref{sec:previous}.

\end{lemma}

\begin{proof}

Since $\ld$ is linear with respect to $\varphi$, it is enough to show inequality (\ref{th:MST}) when $\varphi=\varphi_{k}$ for some positive integer $k$, in which case
\begin{equation}
\ld(\varphi_{k},u,(a,b))=\dint_{I_{k}(\delta,u,(a,b))}\frac{\delta}{(y-x)^{2}}\,dx\,dy,
\nonumber
\end{equation}
where
\begin{equation}
I_{k}(\delta,u,(a,b)):=\left\{(x,y)\in(a,b)^{2}:|u(y)-u(x)|>k\delta\right\}.
\nonumber
\end{equation}

Let us examine the effects on $\ld$ of the three operations performed on $u$. The arguments are the same as in the first part of the proof of~\cite[Proposition~3.4]{AGMP:log-2}, where however only the case of $\varphi_{1}$ was considered.

\subparagraph{\textmd{\textit{Truncation}}}

For every $x$ and $y$ in $(a,b)$ it turns out that
$$|T_{A,B}u(y)-T_{A,B}u(x)|>k\delta
\quad\Longrightarrow\quad
|u(y)-u(x)|>k\delta.$$

This implies that $I_{k}(\delta,T_{A,B}u,(a,b))\subseteq I_{k}(\delta,u,(a,b))$, and therefore
\begin{equation}
\ld(\varphi_{k},u,(a,b))\geq\ld(\varphi_{k},T_{A,B}u,(a,b)).
\label{ineq:T}
\end{equation}

\subparagraph{\textmd{\textit{Vertical $\delta$-segmentation}}}

For every $x$ and $y$ in $(a,b)$ it turns out that
$$|S_{\delta}u(y)-S_{\delta}u(x)|>k\delta
\ \Longrightarrow\ 
|S_{\delta}u(y)-S_{\delta}u(x)|\geq (k+1)\delta
\ \Longrightarrow\ 
|u(y)-u(x)|>k\delta.$$

As before this implies that 
\begin{equation}
\ld(\varphi_{k},T_{A,B}u,(a,b))\geq\ld(\varphi_{k},S_{\delta}T_{A,B}u,(a,b)).
\label{ineq:ST}
\end{equation}

\subparagraph{\textmd{\textit{Nondecreasing rearrangement}}}

We claim that
\begin{equation}
\ld(\varphi_{k},S_{\delta}T_{A,B}u,(a,b))\geq\ld(\varphi_{k},MS_{\delta}T_{A,B}u,(a,b)).
\label{ineq:MST}
\end{equation}

This inequality, together with (\ref{ineq:T}) and (\ref{ineq:ST}), completes the proof of (\ref{th:MST}).

In order to prove (\ref{ineq:MST}), we rely on the theory of semi-discrete arrangements. To this end, we consider the semi-discrete arrangement $v_{\delta}:(a,b)\to\z$ defined by
\begin{equation}
v_{\delta}(x):=\frac{1}{\delta}S_{\delta}T_{A,B}u(x)
\qquad
\forall x\in(a,b)
\label{defn:vd}
\end{equation}
(we recall that $S_{\delta}T_{A,B}u$ takes its values in $\delta\z$, and hence $v_{\delta}(x)$ is integer valued), and the hostility function $c:(0,b-a)\to\re$ defined by $c(\sigma):=\delta\sigma^{-2}$. We observe that 
$$MS_{\delta}T_{A,B}u(x)=\delta Mv_{\delta}(x)
\qquad
\forall x\in(a,b),$$
where $Mv_{\delta}$ is the nondecreasing rearrangement of $v_{\delta}$. From (\ref{defn:vd}) and (\ref{dino:idphi}) it turns out that 
\begin{equation}
\ld(\varphi_{k},S_{\delta}T_{A,B}u,(a,b))=\delta\Lambda_{1}(\varphi_{k},v_{\delta},(a,b))=\mathcal{F}_{k}(c,v_{\delta}),
\nonumber
\end{equation}
and similarly
\begin{equation}
\ld(\varphi_{k},MS_{\delta}T_{A,B}u,(a,b))=\delta\Lambda_{1}(\varphi_{k},Mv_{\delta},(a,b))=\mathcal{F}_{k}(c,Mv_{\delta}),
\nonumber
\end{equation}
so that now (\ref{ineq:MST}) follows from Theorem~\ref{thm:dino-sd}.
\end{proof}


As a second simplifying step, we show that the double integral over $(a,b)^{2}$ can be replaced, in the computation of the liminf, by a double integral over an infinite strip, which is easier to handle. To this end, we introduce the family of functionals
\begin{equation}
\ldh(\varphi,u,(c,d)):=\int_{c}^{d}dx\int_{-\infty}^{+\infty}\varphi\left(\frac{|u(y)-u(x)|}{\delta}\right)\frac{\delta}{(y-x)^{2}}\,dy,
\nonumber
\end{equation}
and we prove the following result.

\begin{lemma}[Extension to a vertical strip]\label{lemma:vertical}

Let $a<c<d<b$ be real numbers, and let $\varphi:[0,+\infty)\to[0,+\infty)$ be a bounded measurable function. For every $\delta>0$, let $\ud:(a,b)\to\re$ be a measurable function. Let us extend $\ud$ to the whole real line by setting $\ud(x)=0$ for every $x\not\in(a,b)$.

Then it turns out that
\begin{equation}
\liminf_{\delta\to 0^{+}}\ld(\varphi,\ud,(a,b))\geq\liminf_{\delta\to 0^{+}}\ldh(\varphi,\ud,(c,d)).
\label{th:extension}
\end{equation}

\end{lemma}

\begin{proof}

Let us set for shortness
\begin{equation}
\fdup(x,y):=\varphi\left(\frac{|\ud(y)-\ud(x)|}{\delta}\right)\frac{\delta}{(y-x)^{2}}
\qquad
\forall (x,y)\in(a,b)\times\re.
\nonumber
\end{equation}

Since $\varphi$ is nonnegative and $(c,d)\subseteq(a,b)$, for every $\delta>0$ it turns out that
\begin{eqnarray}
\ld(\varphi,\ud,(a,b))  &  \geq  &  \int_{c}^{d}dx\int_{a}^{b}\fdup(x,y)\,dy  \nonumber  \\[1ex]
  &  =  &  \ldh(\varphi,\ud,(c,d))-\int_{c}^{d}dx\int_{\re\setminus[a,b]}\fdup(x,y)\,dy.
  \label{ineq:ld-ldh}
\end{eqnarray}

From the boundedness of $\varphi$ it follows that
\begin{equation}
\int_{c}^{d}dx\int_{b}^{+\infty}\fdup(x,y)\,dy\leq\delta\|\varphi\|_{\infty}\int_{c}^{d}dx\int_{b}^{+\infty}\frac{1}{(y-x)^{2}}\,dy.
\nonumber
\end{equation}

Since $d<b$, the double integral in the right-hand side is convergent, and hence
\begin{equation}
\lim_{\delta\to 0^{+}}\int_{c}^{d}dx\int_{b}^{+\infty}\fdup(x,y)\,dy=0.
\label{lim:b-infty}
\end{equation}

In an analogous way we obtain that
\begin{equation}
\lim_{\delta\to 0^{+}}\int_{c}^{d}dx\int_{-\infty}^{a}\fdup(x,y)\,dy=0.
\label{lim:a-infty}
\end{equation}

At this point, (\ref{th:extension}) follows from (\ref{ineq:ld-ldh}), (\ref{lim:b-infty}), and (\ref{lim:a-infty}).
\end{proof}


We are now ready to state and prove the main result of this section.

\begin{prop}[Limit cost of vertical oscillations]\label{prop:main}

Let $(a,b)\subseteq\re$ be an interval, let $\{\ud\}_{\delta>0}\subseteq L^{1}((a,b))$ be a family of functions, let $\varphi\in\pca$ be a piecewise constant interaction law, let $P_{n,\varphi}$ be the multi-variable function defined by (\ref{defn:pnp}), and let $I_{n}(\varphi)$ be its infimum as in (\ref{defn:Inp}).

Let us assume that there exist two real numbers $A\leq B$ such that 
\begin{equation}
\liminf_{\dzp}\meas\{x\in(a,b):\ud(x)\leq A+\ep\}>0
\qquad
\forall\ep>0,
\label{hp:liminf-meas1}
\end{equation}
and
\begin{equation}
\liminf_{\dzp}\meas\{x\in(a,b):\ud(x)\geq B-\ep\}>0
\qquad
\forall\ep>0.
\label{hp:liminf-meas2}
\end{equation}

Then it turns out that
\begin{equation}
\liminf_{\dzp}\ld(\varphi,\ud,(a,b))\geq (B-A)\cdot\liminf_{n\to +\infty}\frac{I_{n}(\varphi)}{n}.
\label{th:liminf-loc}
\end{equation}

\end{prop}

\begin{proof}

To begin with, we observe that (\ref{th:liminf-loc}) is trivial if $A=B$, or if the left-hand side is infinite. Up to restricting ourselves to a sequence $\delta_{k}\to 0^{+}$, we can also assume that the liminf is actually a limit. Therefore, in the sequel we assume that the left-hand side of (\ref{th:liminf-loc}) is uniformly bounded from above, and that $A<B$.

Let us fix $\ep>0$ such that $4\ep<B-A$. Due to assumptions~(\ref{hp:liminf-meas1}) and~(\ref{hp:liminf-meas2}), there exist $\eta>0$ and $\delta_{0}>0$ such that
\begin{equation}
\meas\{x\in(a,b):\ud(x)\leq A+\ep\}\geq\eta
\qquad
\forall\delta\in(0,\delta_{0}),
\label{meas-a-1}
\end{equation}
\begin{equation}
\meas\{x\in(a,b):\ud(x)\geq B-\ep\}\geq\eta
\qquad
\forall\delta\in(0,\delta_{0}).
\label{meas-b-1}
\end{equation}

Let us consider the modified family $\udh:=MS_{\delta}T_{A,B}\ud$ as in Lemma~\ref{lemma:MST}. From (\ref{meas-a-1}) and (\ref{meas-b-1}) it follows that the nondecreasing function $\udh$ satisfies
\begin{equation}
\udh(x)\leq A+2\ep
\qquad
\forall x\in(a,a+\eta),\quad\forall\delta\in(0,\delta_{1}),
\label{meas-a-3}
\end{equation}
\begin{equation}
\udh(x)\geq B-2\ep
\qquad
\forall x\in(b-\eta,b),\quad\forall\delta\in(0,\delta_{1}),
\label{meas-b-3}
\end{equation}
where $\delta_{1}:=\min\{\ep,\delta_{0}\}$. Moreover, from Lemma~\ref{lemma:MST} and Lemma~\ref{lemma:vertical} it follows that
\begin{equation}
\liminf_{\dzp}\ld(\varphi,\ud,(a,b))\geq\liminf_{\dzp}\ld(\varphi,\udh,(a,b))\geq\liminf_{\dzp}\ldh(\varphi,\udh,(a+\eta,b-\eta)),
\nonumber
\end{equation}
where in the computation of the latter we imagine that $\udh$ has been extended to the whole real line by setting it equal to 0 (or any other value) outside $(a,b)$.

In order to compute the last liminf, we need a deeper description of the structure of $\udh$. We know that $\udh$ is nondecreasing and that its image is contained in $\delta\z$. Let $\delta m_{0}$ denote the value of $\udh$ in a right neighborhood of $a$, let us set $x_{0}=a$, and for every positive integer $i$ let us define
\begin{equation}
x_{i}:=\sup\left\{x\in(a,b):\udh(x)<(m_{0}+i)\delta\right\}.
\nonumber
\end{equation}

The sequence $x_{i}$ is nondecreasing, and $x_{i}=b$ for every large enough index $i$. If $x_{i+1}>x_{i}$ for some index $i$, then it turns out that
\begin{equation}
\udh(x)=(m_{0}+i)\delta
\qquad
\forall x\in(x_{i},x_{i+1}).
\nonumber
\end{equation}

If $x_{i+1}=x_{i}$ for some index $i$, this means that
\begin{equation}
\operatorname{meas}\left\{x\in(a,b):\udh(x)=(m_{0}+i)\delta\right\}=0.
\nonumber
\end{equation}

Let $\alpha$ and $\beta$ be the two indices (which of course do depend on $\delta$) such that
\begin{equation}
\alpha:=\min\{i\in\n:x_{i}\geq a+\eta\},
\qquad\mbox{and}\qquad
\beta:=\max\{i\in\n:x_{i}\leq b-\eta\}.
\nonumber
\end{equation}

Let us consider now the interaction law $\varphi$, which we assumed of the form (\ref{defn:phi-pca}), and let $\mu(\varphi)$ denote the smallest index $k\leq m$ such that $\lambda_{k}>0$. To begin with, we show that $x_{i+\mu(\varphi)}>x_{i}$ for every index $i$ with $\alpha\leq i\leq\beta$. Indeed, if this is not the case, then it turns out that
\begin{equation}
\udh(y)-\udh(x)\geq(\mu(\varphi)+1)\delta
\qquad
\forall(x,y)\in(a,x_{i})\times(x_{i},b),
\nonumber
\end{equation}
and in particular
\begin{equation}
\ld(\varphi,\udh,(a,b))\geq\lambda_{\mu(\varphi)}\ld(\varphi_{\mu(\varphi)},\udh,(a,b))\geq\lambda_{\mu(\varphi)}\int_{a}^{x_{i}}dx\int_{x_{i}}^{b}\frac{\delta}{(y-x)^{2}}\,dy,
\nonumber
\end{equation}
which is absurd because the left-hand side is uniformly bounded from above, while the double integral in the right-hand side is divergent.

Let us consider now an integer $k\in\{\mu(\varphi),\ldots,m\}$, and for every $x\in(a,b)$ let us set
\begin{equation}
H_{k,+}(x):=\int_{x}^{+\infty}\varphi_{k}\left(\frac{|\udh(y)-\udh(x)|}{\delta}\right)\frac{\delta}{(y-x)^{2}}\,dy,
\nonumber
\end{equation}
\begin{equation}
H_{k,-}(x):=\int_{-\infty}^{x}\varphi_{k}\left(\frac{|\udh(y)-\udh(x)|}{\delta}\right)\frac{\delta}{(y-x)^{2}}\,dy.
\nonumber
\end{equation}

With this notation it turns out that
\begin{eqnarray}
\ldh\left(\varphi_{k},\udh,(a+\eta,b-\eta)\right)  & =  &  \int_{a+\eta}^{b-\eta}dx\int_{-\infty}^{+\infty}\varphi_{k}\left(\frac{|\udh(y)-\udh(x)|}{\delta}\right)\frac{\delta}{(y-x)^{2}}\,dy \nonumber \\[1ex]
 & = & \int_{a+\eta}^{b-\eta}\left(H_{k,+}(x)+H_{k,-}(x)\strut\right)\,dx \nonumber\\[1ex]
 & \geq & \int_{x_{\alpha}}^{x_{\beta-k}}H_{k,+}(x)\,dx+\int_{x_{\alpha+k}}^{x_{\beta}}H_{k,-}(x)\,dx.
 \label{int:H+-}
\end{eqnarray}

Let us compute the last two integrals separately. For every index $i\in\{\alpha+1,\ldots,\beta-k\}$ it turns out that
\begin{equation}
H_{k,+}(x)=\int_{x_{i+k}}^{+\infty}\frac{\delta}{(y-x)^{2}}\,dy=\frac{\delta}{x_{i+k}-x}
\qquad
\forall x\in(x_{i-1},x_{i}).
\nonumber
\end{equation}

The previous equality assumes that $x_{i-1}<x_{i}$, but actually it is true for trivial reasons also if $x_{i-1}=x_{i}$. It follows that
\begin{equation}
\int_{x_{i-1}}^{x_{i}}H_{k,+}(x)\,dx=\delta\log\frac{x_{i+k}-x_{i-1}}{x_{i+k}-x_{i}}\nonumber
\end{equation}
for every $i\in\{\alpha+1,\ldots,\beta-k\}$, and therefore
\begin{equation}
\int_{x_{\alpha}}^{x_{\beta-k}}H_{k,+}(x)\,dx=\sum_{i=\alpha+1}^{\beta-k}\int_{x_{i-1}}^{x_{i}}H_{k,+}(x)\,dx=\delta\sum_{i=\alpha+1}^{\beta-k}\log\frac{x_{i+k}-x_{i-1}}{x_{i+k}-x_{i}}.
\label{int:H+}
\end{equation}

In an analogous way, for every index $i\in\{\alpha+k+1,\ldots,\beta\}$ it turns out that
\begin{equation}
H_{k,-}(x)=\int_{-\infty}^{x_{i-k-1}}\frac{\delta}{(y-x)^{2}}\,dy=\frac{\delta}{x-x_{i-k-1}}
\qquad
\forall x\in(x_{i-1},x_{i}),
\nonumber
\end{equation}
so that, with a shift of indices, we obtain that
\begin{equation}
\int_{x_{i+k-1}}^{x_{i+k}}H_{k,-}(x)\,dx=\delta\log\frac{x_{i+k}-x_{i-1}}{x_{i+k-1}-x_{i-1}}\nonumber
\end{equation}
for every $i\in\{\alpha+1,\ldots,\beta-k\}$, and therefore
\begin{equation}
\int_{x_{\alpha+k}}^{x_{\beta}}H_{k,-}(x)\,dx=\sum_{i=\alpha+1}^{\beta-k}\int_{x_{i+k-1}}^{x_{i+k}}H_{k,-}(x)\,dx=\delta\sum_{i=\alpha+1}^{\beta-k}\log\frac{x_{i+k}-x_{i-1}}{x_{i+k-1}-x_{i-1}}.
\label{int:H-}
\end{equation}

Plugging (\ref{int:H+}) and (\ref{int:H-}) into (\ref{int:H+-}), we find that
\begin{equation}
\ldh\left(\varphi_{k},\udh,(a+\eta,b-\eta)\right)\geq\delta\sum_{i=\alpha+1}^{\beta-k}\log\frac{(x_{i+k}-x_{i-1})^{2}}{(x_{i+k-1}-x_{i-1})(x_{i+k}-x_{i})}.
\nonumber
\end{equation}

Setting $\ell_{i}:=x_{\alpha+i}-x_{\alpha+i-1}$ for every $i\in\{1,\ldots,\beta-\alpha\}$, we can write the last inequality in the form
\begin{eqnarray}
\ldh\left(\varphi_{k},\udh,(a+\eta,b-\eta)\right)  & \geq & \delta\sum_{i=1}^{\beta-\alpha-k}\log\frac{(\ell_{i}+\ldots+\ell_{i+k})^{2}}{(\ell_{i}+\ldots+\ell_{i+k-1})(\ell_{i+1}+\ldots+\ell_{i+k})} \nonumber \\[1ex]
 & = & \delta L_{k}(\ell_{1},\ldots,\ell_{\beta-\alpha}),
\nonumber
\end{eqnarray}
where $L_{k}$ is the multi-variable function defined in (\ref{defn:Lk}). We observe that the denominators do not vanish because $k\geq\mu(\varphi)$, and we have already proved that $x_{i+\mu(\varphi)}>x_{i}$, which is equivalent to saying that the list $(\ell_{1},\ldots,\ell_{\beta-\alpha})$ contains no $k$ consecutive terms that vanish. Since $\ldh$ is linear with respect to $\varphi$, we deduce that
\begin{eqnarray}
\ldh\left(\varphi,\udh,(a+\eta,b-\eta)\right) & \geq & \delta\sum_{k=\mu(\varphi)}^{m}\lk L_{k}(\ell_{1},\ldots,\ell_{\beta-\alpha}) 
\nonumber  \\[1ex]
& = & \delta P_{\beta-\alpha,\varphi}(\ell_{1},\ldots,\ell_{\beta-\alpha}) 
\nonumber \\[1ex]
& \geq & \delta I_{\beta-\alpha}(\varphi).
\nonumber
\end{eqnarray}

Letting $\dzp$, and observing that $\beta-\alpha\to +\infty$, we conclude that
\begin{eqnarray}
\liminf_{\dzp}\ldh\left(\varphi,\udh,(a+\eta,b-\eta)\right) & \geq & \liminf_{\dzp}\delta I_{\beta-\alpha}(\varphi) 
\nonumber  \\[1ex]
& \geq & \liminf_{\dzp}\delta(\beta-\alpha)\cdot\liminf_{\dzp}\frac{I_{\beta-\alpha}(\varphi)}{\beta-\alpha}  
\nonumber  \\[1ex]
& \geq & \liminf_{\dzp}\delta(\beta-\alpha)\cdot\liminf_{n\to +\infty}\frac{I_{n}(\varphi)}{n}. 
\label{est:almost-done}
\end{eqnarray}

It remains to compute the liminf of $\delta(\beta-\alpha)$. To this end, from (\ref{meas-a-3}) and the minimality of $\alpha$ we deduce that
$$A+2\ep\geq\udh(x)=(m_{0}+\alpha-1)\delta
\qquad
\forall x\in(x_{\alpha-1},x_{\alpha}).$$

Similarly, from (\ref{meas-b-3}) and the maximality of $\beta$ we deduce that
$$B-2\ep\leq\udh(x)=(m_{0}+\beta)\delta
\qquad
\forall x\in(x_{\beta},x_{\beta+1}).$$

It follows that $(\beta-\alpha)\delta\geq B-A-4\ep-\delta$, and therefore from (\ref{est:almost-done}) we conclude that
\begin{equation}
\liminf_{\dzp}\ldh\left(\varphi,\udh,(a+\eta,b-\eta)\right)\geq(B-A-4\ep)\cdot\liminf_{n\to +\infty}\frac{I_{n}(\varphi)}{n}.
\nonumber
\end{equation}

Letting $\ep\to 0^{+}$, we finally deduce (\ref{th:liminf-loc}).
\end{proof}

As observed in~\cite{AGMP:log-2}, we can rewrite Proposition~\ref{prop:main} as a relation between the liminf of the energy and the liminf of oscillations, as follows.
\begin{cor}\label{cor:osc}

Let $(a,b)$, $\ud$, $\varphi$ and $I_{n}(\varphi)$ be as in Proposition~\ref{prop:main}. For every $\delta>0$, let $\operatorname{osc}(\ud,(a,b))$ denote the essential oscillation of $\ud$ in $(a,b)$.

Then it turns out that
\begin{equation}
\liminf_{\dzp}\ld(\varphi,\ud,(a,b))\geq\left(\liminf_{\dzp}\operatorname{osc}(\ud,(a,b))\right)\cdot\liminf_{n\to +\infty}\frac{I_{n}(\varphi)}{n}.
\nonumber
\end{equation}

\end{cor}

\begin{proof}

Let $i_{\delta}$ and $s_{\delta}$ denote the essential infimum and the essential supremum of $\ud(x)$ in $(a,b)$, respectively. Let us assume that $i_{\delta}$ and $s_{\delta}$ are real numbers (otherwise an analogous argument works with standard minor changes). Let us set $w_{\delta}(x):=\ud(x)-i_{\delta}$, and let us observe that
$$\ld(\varphi,\ud,(a,b))=\ld(\varphi,w_{\delta},(a,b))
\qquad
\forall \delta>0,$$
and that $w_{\delta}$ satisfies (\ref{hp:liminf-meas1}) and (\ref{hp:liminf-meas2}) with $A:=0$ and 
$$B:=\liminf_{\dzp}(s_{\delta}-i_{\delta})=\liminf_{\dzp}\operatorname{osc}(\ud,(a,b)).$$
 At this point the conclusion follows from Proposition~\ref{prop:main}.
\end{proof}

Combining Theorem~\ref{prop:loc2glob} and Corollary~\ref{cor:osc}, we obtain the following result, which connects the Gamma-liminf of the family (\ref{defn:idphi}) to the multi-variable minimum problems of Section~\ref{sec:min-pbm}.

\begin{prop}\label{prop:g-liminf-pca}

For every positive integer $d$ and every interaction law $\varphi\in\pca$ it turns out that
\begin{equation}
\gliminf_{\dzp}\ld\left(\varphi,u,\re^{d}\right)\geq G_{d}\cdot\frac{1}{2}\liminf_{n\to+\infty}\frac{I_{\varphi}(n)}{n}\cdot\lz\left(u,\re^{d}\right)
\qquad
\forall u\in L^{1}(\re^{d}).
\nonumber
\end{equation}
\end{prop}


\setcounter{equation}{0}
\section{Proof of our main results}\label{sec:proofs}

\subsubsection*{Proof of Theorem~\ref{thm:Q4}}

Let us consider, for every positive integer $m$, the interaction law defined as
\begin{equation}
\psi_{m}(t):=\sum_{k=1}^{2^{m}-1}\varphi_{k}(t)
\qquad
\forall t\geq 0.
\label{defn:psim}
\end{equation}

This interaction law can be written in the form (\ref{defn:pca2}) with $a_{1}=\ldots=a_{m}=1$, and hence $\psi_{m}\in\pca_{2}$. As a consequence, from Proposition~\ref{prop:package} we deduce that
\begin{equation}
\liminf_{n\to +\infty}\frac{I_{n}(\psi_{m})}{n}\geq m\cdot 2\log 2,
\nonumber
\end{equation}
and therefore from Proposition~\ref{prop:g-liminf-pca} we obtain that
\begin{equation}
\gliminf_{\dzp}\ld\left(\psi_{m},u,\re^{d}\right)\geq G_{d}\cdot m\log 2\cdot\lz\left(u,\re^{d}\right).
\label{g-liminf-Q4-1}
\end{equation}

On the other hand, from (\ref{thmbibl:g-conv}) we know that
\begin{equation}
\glim_{\dzp}\ld\left(\psi_{m},u,\re^{d}\right)= G_{d}\cdot N(\psi_{m})\cdot K_{d}(\psi_{m})\cdot\lz\left(u,\re^{d}\right),
\label{g-liminf-Q4-2}
\end{equation}
and with some simple calculus we find that
\begin{equation}
N(\psi_{m})=\int_{0}^{+\infty}\frac{\psi_{m}(t)}{t^{2}}\,dt=\sum_{k=1}^{2^{m}-1}\int_{0}^{+\infty}\frac{\varphi_{k}(t)}{t^{2}}\,dt=\sum_{k=1}^{2^{m}-1}\int_{k}^{+\infty}\frac{1}{t^{2}}\,dt=\sum_{k=1}^{2^{m}-1}\frac{1}{k}.
\label{est:N-Q4}
\end{equation}

Comparing (\ref{g-liminf-Q4-1}) and (\ref{g-liminf-Q4-2}) we obtain that
\begin{equation}
K_{d}(\psi_{m})\geq\frac{m \log 2}{N(\psi_{m})}
\qquad
\forall m\geq 1.
\label{est:K-Q4}
\end{equation}

Now from (\ref{est:N-Q4}) we know that $N(\psi_{m})\sim m\log 2$ as $m\to+\infty$, and therefore we conclude that $K_{d}(\psi_{m})\to 1 $ as $m\to +\infty$, independently of the space dimension.\qed


\subsubsection*{Proof of Theorem~\ref{thm:Q2} -- Statement~(1)}

Let us consider the interaction law $\theta\in\mathcal{A}_{0}$ defined by
\begin{equation}
\pt(t):=\left\{
\begin{array}{ll}
    0  &  \mbox{if }t\in[0,1],  \\[0.5ex]
    t-1  & \mbox{if }t\in[1,2],   \\[0.5ex]
    1  & \mbox{if }t\geq 2.  
\end{array}\right.
\label{defn:theta}
\end{equation}

We claim that the shape factor of $\theta$ is one in any space dimension.

For every positive integer $m$ we consider the interaction law
\begin{equation}
\pt_{m}(t):=\sum_{k=2^{m-1}}^{2^{m}-1}\varphi_{k}(t),
\nonumber
\end{equation}
and the rescaled function
\begin{equation}
\pth_{m}(t):=\frac{1}{2^{m-1}}\pt_{m}\left((2^{m-1}-1)t\right).
\nonumber
\end{equation}

To begin with, we show that
\begin{equation}
\pt(t)\geq\pth_{m}(t)
\qquad
\forall t\geq 0,\quad\forall m\geq 1.
\label{ineq:pt-pth}
\end{equation}

To this end, we distinguish three cases.

\begin{itemize}

\item If $t\in[0,1]$, then $(2^{m-1}-1)t\leq 2^{m-1}-1$, and hence $\varphi_{k}\left((2^{m-1}-1)t\right)=0$ for every $k\geq 2^{m-1}$. It follows that $\pth_{m}(t)=0$, and hence (\ref{ineq:pt-pth}) is trivial.

\item  If $t\geq 2$, then
$$\pth_{m}(t)\leq\frac{1}{2^{m-1}}\sum_{k=2^{m-1}}^{2^{m}-1}1=1=\pt(t),$$
and therefore (\ref{ineq:pt-pth}) is again satisfied.

\item If $t\in(1,2)$, let us choose $i\in\{2^{m-1},\ldots,2^{m}-1\}$ such that
$$\frac{i}{2^{m-1}}<t\leq\frac{i+1}{2^{m-1}}.$$

Since
$$(2^{m-1}-1)t\leq(2^{m-1}-1)\cdot\frac{i+1}{2^{m-1}}=i-\frac{i+1-2^{m-1}}{2^{m-1}}\leq i,$$
we deduce that
$$\varphi_{k}\left((2^{m-1}-1)t\right)=0
\qquad
\forall k\geq i,$$
and therefore
$$\pth_{m}(t)\leq\frac{1}{2^{m-1}}\sum_{k=2^{m-1}}^{i-1}\varphi_{k}(t)\leq\frac{1}{2^{m-1}}\left(i-2^{m-1}\right)=\frac{i}{2^{m-1}}-1\leq t-1=\pt(t),$$
which proves (\ref{ineq:pt-pth}) also in this case.
 
\end{itemize}

From inequality (\ref{ineq:pt-pth}), and the rescaling properties of $\ld$ with respect to the interaction law, we deduce that
\begin{eqnarray}
\gliminf_{\dzp}\ld\left(\pt,u,\re^{d}\right) & \geq & \gliminf_{\dzp}\ld\left(\pth_{m},u,\re^{d}\right) \nonumber \\[1ex]
 & = & \frac{2^{m-1}-1}{2^{m-1}}\cdot\gliminf_{\dzp}\ld\left(\theta_{m},u,\re^{d}\right).
 \label{g-liminf-Q2-0}
\end{eqnarray}

Now we observe that the interaction law $\pt_{m}(t)$ can be written in the form (\ref{defn:pca2}) with $a_{1}=\ldots=a_{m-1}=0$ and $a_{m}=1$, and hence $\pt_{m}\in\pca_{2}$. As a consequence, from Proposition~\ref{prop:package} we deduce that
\begin{equation}
\liminf_{n\to +\infty}\frac{I_{n}(\pt_{m})}{n}\geq 2\log 2,
\nonumber
\end{equation}
and therefore from Proposition~\ref{prop:g-liminf-pca} we obtain that
\begin{equation}
\gliminf_{\dzp}\ld\left(\pt_{m},u,\re^{d}\right)\geq G_{d}\cdot \log 2\cdot\lz\left(u,\re^{d}\right).
\nonumber
\end{equation}

Plugging this estimate into (\ref{g-liminf-Q2-0}), and letting $m\to+\infty$, we deduce that
\begin{equation}
\gliminf_{\dzp}\ld\left(\pt,u,\re^{d}\right)\geq G_{d}\cdot \log 2\cdot\lz\left(u,\re^{d}\right).
\label{g-liminf-Q2-1}
\end{equation}

On the other hand, from (\ref{thmbibl:g-conv}) we know that
\begin{equation}
\glim_{\dzp}\ld\left(\pt,u,\re^{d}\right)= G_{d}\cdot N(\pt)\cdot K_{d}(\pt)\cdot\lz\left(u,\re^{d}\right),
\label{g-liminf-Q2-2}
\end{equation}
and with some simple calculus we find that
\begin{equation}
N(\theta)=\int_{1}^{2}\frac{t-1}{t^{2}}\,dt+\int_{2}^{+\infty}\frac{1}{t^{2}}\,dt=\log 2.
\nonumber
\end{equation}

Comparing (\ref{g-liminf-Q2-1}) and (\ref{g-liminf-Q2-2}) we conclude that $K_{d}(\pt)=1$ in any space dimension.


\subsubsection*{Proof of Theorem~\ref{thm:Q2} -- Statement~(2)}

The function $\zeta$ is continuous because $f(z)\to 0$ as $z\to -\infty$. It is also bounded and monotone due to the corresponding assumptions on $f(z)$. Finally, assumption (\ref{hp:f-quadratic}) implies the existence of a constant $a$ such that $\zeta(t)\leq at^{2}$ for every $t\geq 0$. This proves that $\zeta\in\mathcal{A}$.

In order to compute scale and shape factor of $\zeta$, we observe that it can be written in the form
$$\zeta(t)=\sum_{z=-\infty}^{+\infty}(f(z+1)-f(z))\cdot\theta(2^{-z}t)
\qquad
\forall t\geq 0,$$
where $\theta$ is the interaction law defined in (\ref{defn:theta}). Due to the additivity and to the rescaling properties of $\ld$ with respect to the interaction law, from this representation it follows that
\begin{eqnarray}
\hspace{-2em}\glim_{\dzp}\ld\left(\zeta,u,\re^{d}\right) & \geq & \sum_{z=-\infty}^{+\infty}(f(z+1)-f(z))2^{-z}\cdot\gliminf_{\dzp}\ld\left(\pt,u,\re^{d}\right) \nonumber \\
 & = &  \sum_{z=-\infty}^{+\infty}(f(z+1)-f(z))2^{-z}\cdot G_{d}\cdot N(\pt)\cdot\lz\left(u,\re^{d}\right),
\label{g-liminf-Q22-1}
\end{eqnarray}
where in the last equality we have exploited that $K_{d}(\pt)=1$.

On the other hand, from (\ref{thmbibl:g-conv}) we know that
\begin{equation}
\glim_{\dzp}\ld\left(\zeta,u,\re^{d}\right)= G_{d}\cdot N(\zeta)\cdot K_{d}(\zeta)\cdot\lz\left(u,\re^{d}\right).
\label{g-liminf-Q22-2}
\end{equation}

Since
\begin{equation}
N(\zeta)=N(\theta)\cdot\sum_{z=-\infty}^{+\infty}(f(z+1)-f(z))2^{-z},
\nonumber
\end{equation}
comparing (\ref{g-liminf-Q22-1}) and (\ref{g-liminf-Q22-2}) we conclude that $K_{d}(\zeta)=1$ in any space dimension.\qed

\setcounter{equation}{0}
\section{Final remarks}\label{sec:speculation}

In this section we present some variants of our main results, and we speculate about some possible future extensions of the theory developed in this paper.

\paragraph{A counterexample to the short-range question with strict inequalities}

Let us consider the interaction laws
\begin{equation}
\varphi_{\ep}(t):=\left\{\!\!
\begin{array}{l@{\quad}l}
  c_{1,\ep}\cdot \ep t^{2}   &  \mbox{if }t\in[0,1], \\[0.5ex]
  c_{1,\ep}    &  \mbox{if }t> 1,
\end{array}\right.
\qquad\qquad
\psi(t):=c_{2}\psi_{2}(t),
\nonumber
\end{equation}
where $\psi_{2}(t)$ is defined by (\ref{defn:psim}) with $m=2$, and the constants $c_{1,\ep}$ and $c_{2}$ are chosen in such a way that $N(\varphi_{\ep})=N(\psi)=1$. 

From (\ref{est:K-Q4}) and (\ref{est:N-Q4}) with $m=2$ it follows that $K_{d}(\psi)\geq (12/11)\log 2$. On the other hand, it is possible (but not completely trivial) to show that $K_{d}(\varphi_{\ep})\to\log 2$ as $\ep\to 0^{+}$.

Therefore, when $\ep$ is small enough, this is an example of two interaction laws $\varphi_{\ep}$ and $\psi$ with equal scale factor, satisfying $\varphi_{\ep}(t)>\psi(t)$ for every $t\in(0,1]$, but nevertheless $K_{d}(\varphi_{\ep})<K_{d}(\psi)$ in every space dimension. This provides a  counterexample to Question~\ref{Q4} with strict inequalities.

\paragraph{True Gamma-limits and smooth recovery families}

In this paper we limited ourselves to providing estimates from below for the Gamma-liminf, since they are enough to establish both Theorem~\ref{thm:Q4} and Theorem~\ref{thm:Q2}. On the other hand, with little further effort we could prove that actually the lower bound coincides with the Gamma-limit.

This is evident in the case of the interaction laws with shape factor equal to one, for example all those provided by statement~(2) of Theorem~\ref{thm:Q2}, because for them the pointwise limit coincides on smooth functions with the estimate from below for the Gamma-liminf. Therefore, for all these interaction laws we now know both the Gamma-limit with exact values of the constants in any space dimension, and the existence of smooth recovery families.

As for the interaction laws $\psi_{m}(t)$ defined by (\ref{defn:psim}), again we can show that the Gamma-limit coincides with the lower bound we obtained for the Gamma-liminf, namely 
\begin{equation}
\glim_{\dzp}\ld(\psi_{m},u,\re^{d})=G_{d}\cdot m\log 2\cdot\lz(u,\re^{d}).
\label{cong:g-lim-psim}
\end{equation}

In order to prove this result, one should follow the path we pursued in~\cite{AGMP:log-2}. The main idea is that in any space dimension the family $S_{\delta}u$ of vertical $\delta$-segmentations of $u$ is a recovery family when $u$ is piecewise $C^{1}$ or piecewise affine with compact support, and those classes are dense in energy for the right-hand side of (\ref{cong:g-lim-psim}). Since vertical $\delta$-segmentations of piecewise affine functions with compact support are step functions with level sets that are finite unions of polytopes, it is enough to further approximate them in order to produce recovery sequences made by functions of class $C^{\infty}$ with compact support. We refer to~\cite{AGMP:log-2} for the details. Therefore, also in the case of the interaction laws $\psi_{m}(t)$, we end up with a Gamma-convergence result with both exact values of the constants in any space dimension, and existence of smooth recovery families.

The same argument should work for all interaction laws in $\pca_{2}$, and more generally for all interaction laws $\varphi\in\pca$ for which $I_{n}(\varphi)$ is realized asymptotically when all the variables are equal.

\paragraph{Toward a general formula for the Gamma-limit}

We suspect that the lower bound in Proposition~\ref{prop:main} might be optimal for every $\varphi\in\pca$, and that the liminf in the right-hand side of (\ref{th:liminf-loc}) is actually a limit. Thanks to Proposition~\ref{prop:g-liminf-pca}, this would imply that
\begin{equation}
\glim_{\dzp}\ld(\varphi,u,\re^{d})=G_{d}\cdot\frac{1}{2}\lim_{n\to +\infty}\frac{I_{n}(\varphi)}{n}\cdot\lz(u,\re^{d})
\qquad
\forall u\in L^{1}(\re^{d})
\label{cong:g-lim}
\end{equation}
for every $\varphi\in\pca$. In order to prove this result, the main difficulty seems to be the construction of recovery sequences, which in general can no more be obtained simply by vertical $\delta$-segmentation. On the contrary, the construction should now take into account the pattern that realizes the infimum $I_{n}(\varphi)$.

A representation of the form (\ref{cong:g-lim}), if true, would be important because any interaction law can be approximated from below by piecewise constant interaction laws with steps of equal horizontal length (as we did in the proof of statement~(1) of Theorem~\ref{thm:Q2}), and these laws are rescalings of laws in $\pca$.

This kind of representation would be even more important if it were true that the Gamma-limit of $\ld(\varphi,u,\re^{d})$ is the supremum of the Gamma-limits of $\ld(\psi,u,\re^{d})$ as $\psi$ varies in the set of all piecewise constant interaction laws, with steps of equal horizontal length, that are less than or equal to $\varphi$. A confirmation of this conjecture would open the way for answering several questions raised in~\cite{2018-AnPDE-BreNgu}: a simplified proof of the Gamma-convergence result in full generality, a less implicit formula for shape factors, and existence of smooth recovery families.

\paragraph{Characterization of interaction laws without gap} Concerning the gap between the pointwise limit and the Gamma-limit, the challenge is now characterizing all interaction laws with shape factor equal to one. Let us summarize what we know for the time being on this specific issue.

\begin{itemize}

\item Continuity does not guarantee the lack of the gap. Among continuous interaction laws, we have both examples without gap (all interaction laws provided by Theorem~\ref{thm:Q2}), and interaction laws with gap. Indeed, it is possible to show that the piecewise affine interaction law that is equal to 0 for $t\in[0,1-\ep]$, and equal to 1 for $t\geq 1$, has a shape factor that tends to $\log 2$ as $\ep\to 0^{+}$. 

Conversely, we have no example of discontinuous interaction law without gap.

\item  It is not a matter of vanishing in a neighborhood of the origin. Among the interactions laws in $\mathcal{A}_{0}$ we have both examples without gap (the interaction law $\theta$ defined in (\ref{defn:theta})), and examples with gap (the model interaction law $\varphi_{1}$). Among the interactions laws that are positive for every $t>0$ we have both examples without gap (defined as in statement~(2) of Theorem~\ref{thm:Q2}), and examples with gap (the interaction law $\varphi_{\ep}(t)$ defined at the beginning of this section).

\item  The shape factor is concave when restricted to interaction laws with equal scale factor. As a consequence, any convex combination of interaction laws with the same scale factor, and shape factor equal to one, has again shape factor equal to one. Considering that now we know many  interaction laws with shape factor equal to one, this leads us to guess that the set of interaction laws with shape factor equal to one might be quite large.

\end{itemize}

\paragraph{More general exponents}

It should not be difficult to extend the results of this paper to the more general family of functionals
\begin{equation}
\Lambda_{\delta,p}(\varphi,u,\Omega):=\dint_{\Omega^{2}}\varphi\left(\frac{|u(y)-u(x)|}{\delta}\right)\frac{\delta}{|y-x|^{d+p}}\,dx\,dy,
\nonumber
\end{equation}
where $p>1$ is a real number. This case has not been fully explored yet (a paper in this direction has been announced in~\cite{2018-AnPDE-BreNgu}), but the Gamma-limit is expected to be a multiple of the $L^{p}$-norm of the gradient of $u$, and the exact constant was found in~\cite{AGMP:log-2} in the case $\varphi=\varphi_{1}$. When extending the results of this paper, the presence of the general exponent $p>1$ requires probably only a change in definition (\ref{defn:Lk}), which now should be replaced by something like
\begin{eqnarray*}
L_{k,p} & := & \int_{S_{i.k}}^{S_{i,k+1}}\frac{1}{\sigma^{p}}\,d\sigma+\int_{S_{i+1.k}}^{S_{i,k+1}}\frac{1}{\sigma^{p}}\,d\sigma
\nonumber \\[1ex]
 & = & \frac{1}{p-1}\sum_{i=1}^{n-k}\left(-\frac{2}{[S_{i,k+1}]^{p-1}}+\frac{1}{[S_{i,k}]^{p-1}}+\frac{1}{[S_{i+1,k}]^{p-1}}\right).
\end{eqnarray*}

Again, the special structure of the terms of the sum should guarantee the telescopic effect as in Lemma~\ref{lemma:telescopic}.


\subsubsection*{\centering Acknowledgments}

We would like to thank H.~Brezis for his positive feedback on the manuscript~\cite{AGMP:log-2}, and for encouraging us to investigate the gap between Gamma-limit and pointwise limit.


\label{NumeroPagine}

\end{document}